\documentclass[10pt]{amsart}
\usepackage{amsrefs}
\usepackage{amssymb}
\usepackage[all]{xy}

\usepackage{hyperref}

\usepackage{tikz}
\usetikzlibrary{matrix,arrows}

\DeclareMathOperator{\Gal}{Gal}

\DeclareMathOperator{\Img}{Im}
\DeclareMathOperator{\Inf}{inf}

\DeclareMathOperator{\Ker}{Ker}

\DeclareMathOperator{\res}{res}

\DeclareMathOperator{\dd}{d}

\DeclareFontFamily{U}{wncy}{}
\DeclareFontShape{U}{wncy}{m}{n}{<->wncyr10}{}
\DeclareSymbolFont{mcy}{U}{wncy}{m}{n}
\DeclareMathSymbol{\Sha}{\mathord}{mcy}{"58}
\DeclareMathSymbol{\sha}{\mathord}{mcy}{"78}

\begin{document}

\newtheorem{thm}{Theorem}[section]
\newtheorem{cor}[thm]{Corollary}
\newtheorem{lem}[thm]{Lemma}
\newtheorem{prop}[thm]{Proposition}
\newtheorem{defin}[thm]{Definition}
\newtheorem{exam}[thm]{Example}
\newtheorem{examples}[thm]{Examples}
\newtheorem{rem}[thm]{Remark}
\newtheorem{case}{\sl Case}
\newtheorem{claim}{Claim}
\newtheorem{question}[thm]{Question}
\newtheorem{conj}[thm]{Conjecture}
\newtheorem*{notation}{Notation}
\swapnumbers
\newtheorem{rems}[thm]{Remarks}
\newtheorem*{acknowledgment}{Acknowledgment}
\newtheorem*{thmno}{Theorem}

\newtheorem{questions}[thm]{Questions}
\numberwithin{equation}{section}

\newcommand{\gr}{\mathrm{gr}}
\newcommand{\inv}{^{-1}}
\newcommand{\isom}{\cong}
\newcommand{\dbC}{\mathbb{C}}
\newcommand{\F}{\mathbb{F}}
\newcommand{\dbN}{\mathbb{N}}
\newcommand{\Q}{\mathbb{Q}}
\newcommand{\dbR}{\mathbb{R}}
\newcommand{\dbU}{\mathbb{U}}
\newcommand{\Z}{\mathbb{Z}}
\newcommand{\calG}{\mathcal{G}}
\newcommand{\K}{\mathbb{K}}
\newcommand{\bfH}{\mathbf{H}}
\newcommand{\bfA}{\mathbf{A}}
\newcommand{\bfT}{\mathbf{T}}
\newcommand{\bfLam}{\mathbf{\Lambda}}


\newcommand{\hac}{\hat c}
\newcommand{\hatheta}{\hat\theta}

\title[1-smooth pro-$p$ groups]{1-smooth pro-$p$ groups \\ and Bloch-Kato pro-$p$ groups}
\author{Claudio Quadrelli}
\address{Department of Mathematics and Applications, University of Milano Bicocca, 20125 Milan, Italy EU}
\email{claudio.quadrelli@unimib.it}
\date{\today}

\begin{abstract}
 Let $p$ be a prime.
 A pro-$p$ group $G$ is said to be 1-smooth if it can be endowed with a homomorphism of pro-$p$ groups $G\to1+p\Z_p$ satisfying a formal version of Hilbert 90.
By Kummer theory, maximal pro-$p$ Galois groups of fields containing a root of 1 of order $p$, together with the cyclotomic character, are 1-smooth.
We prove that a finitely generated $p$-adic analytic pro-$p$ group is 1-smooth if, and only if, it occurs as the maximal pro-$p$ Galois group of a field containing a root of 1 of order $p$.
This gives a positive answer to De~Clerq-Florence's ``Smoothness Conjecture'' --- which states that the bijectivity of the norm residue homorphism (i.e., the Bloch-Kato Conjecture) follows from 1-smoothness --- for the class of finitely generated $p$-adic analytic pro-$p$ groups.
\end{abstract}



\maketitle

\section{Introduction}
\label{sec:intro}

For a field $\K$ let $\bar{\K}_s$ denote the separable closure of $K$, and $G_{\K}=\Gal(\bar {\K}_s/{\K})$ the absolute Galois
group of ${\K}$.
One of the main open questions in modern Galois theory is to describe absolute Galois groups of
fields among profinite groups.
The description of the maximal pro-$p$ Galois group $G_{\K}(p)$ --- i.e., the Galois group
of the maximal $p$-extension ${\K}(p)/{\K}$ --- among pro-$p$ groups, for a given prime number $p$,
is already a challenging task.
One of the oldest known obstructions for the realization of a pro-$p$ group as $G_{\K}(p)$ for some field ${\K}$ 
comes from the Artin-Schreier theorem (whose pro-$p$ version is due to E.~Becker, see \cite{becker}): the only 
non-trivial finite group which occurs as the absolute Galois group (and maximal pro-$p$ Galois group) of a field is the cyclic group
of order two.

The proof of the celebrated Bloch-Kato conjecture, by M.~Rost and V.~Voevodsky (with C.~Weibel's ``patch'', see \cite{rost,voev,weibel,weibel2,HW:book}),
provided a description of the Galois cohomology of absolute Galois groups of fields in terms of low degree cohomology.
In particular, the Norm Residue Theorem (also called the Rost-Voevodsky Theorem) implies that if $\K$ contains a root of 1 of order $p$, then $G_{\K}(p)$ is a {\sl Bloch-Kato pro-$p$ group}, i.e., the $\Z/p$-cohomology algebra of every closed subgroup of $G_{\K}(p)$ is a {\sl quadratic} algebra.
This led to the achievement of new obstructions for the realization of pro-$p$ groups as maximal pro-$p$ Galois groups
(see, e.g., \cite{em,cem,cq:bk,qw:cyc}).
For instance, one may recover the Artin-Schreier obstruction as consequence of the Bloch-Kato property (see, e.g., \cite[p.~796]{cq:bk}).

A pair $\calG=(G,\theta)$, consisting of a pro-$p$ group $G$ together with a morphism of pro-$p$ groups $\theta\colon G\to1+p\Z_p$, is called an {\sl oriented pro-$p$ group} (see \cite{qw:cyc}) --- here $1+p\Z_p$ denotes the multiplicative abelian pro-$p$ group $\{1+p\lambda\mid\lambda\in\Z_p\}$.
Given a field $\K$ containing a primitive $p$-th root of unity of 1, the maximal pro-$p$ Galois group of ${\K}$ may be considered naturally as an oriented pro-$p$ group $\calG_{\K}=(G_{\K}(p),\theta_{\K,p})$, where $\theta_{\K,p}\colon G_{\K}(p)\to1+p\Z_p$ is the {\sl cyclotomic character}, which describes the action of $G_{\K}(p)$ on the roots of $1$ of $p$-power order lying in $\bar \K(p)$ (see \cite[\S~4]{eq:kummer}).

The oriented pro-$p$ group $\calG_{\K}$ satisfies the following formal version of Hilbert 90.
Given an oriented pro-$p$ group $\calG=(G,\theta)$, let $\Z_p(\theta)$ denote the continuous $G$-module which is isomorphic to $\Z_p$ as an abelian pro-$p$ group, and endowed with the left $G$-action defined by $g.v=\theta(g)\cdot v$ for all $g\in G$ and $v\in\Z_p(\theta)$.
The oriented pro-$p$ group $\calG$ is said to be {\sl Kummerian} if the morphism 
\begin{equation}\label{eq:def kummer}
   H^1(G,\Z_p(\theta)/p^n\Z_p(\theta))\longrightarrow H^1(G,\Z_p(\theta)/p\Z_p(\theta)),
\end{equation}
induced by the epimorphism of $G$-modules $\Z_p(\theta)/p^n\Z_p(\theta)\to \Z_p(\theta)/p\Z_p(\theta)$, is surjective for every $n\geq1$; and moreover $\calG$ is said to be {\sl 1-smooth} if the oriented pro-$p$ group $\calG_H=(H,\theta\vert_H)$ is  Kummerian for every closed subgroup $H\subseteq G$.
By Kummer theory, the oriented pro-$p$ group $\calG_{\K}$ is 1-smooth (see \cite[Prop.~14.19]{dcf:lift} and \cite[Thm.~1.1]{qw:cyc}).

In the paper \cite{dcf:lift} --- motivated by the pursuit of an ``explicit'' proof
of the Bloch-Kato conjecture as an alternative to the proof by Voevodsky ---
 C.~De~Clerq and M.~Florence introduce the 1-smoothness property.
In particular, they formulate the ``Smoothness Conjecture'': namely, that it is possible to deduce the
surjectivity part of the Block-Kato conjecture (which is known to be the ``hard part'' of the conjecture) from the fact that the oriented pro-$p$ group $\calG_{\K}$ arising from a field $\K$ containing a root of 1 of order $p$, is 1-smooth: in other words, they conjecture that a 1-smooth oriented pro-$p$ group yields a {\sl weakly Bloch-Kato} pro-$p$ group (i.e., a pro-$p$ group whose $\Z/p$-cohomology satisfies the aforementioned surjectivity feature, see Definition~\ref{defin:BK}).
For example, one has that 1-smoothness implies the Artin-Schreier obstruction (see Example~\ref{ex:AS}).

Our goal is to prove that in the class of finitely generated {\sl $p$-adic analytic pro-$p$ groups}, 1-smoothness implies the Bloch-Kato property and the realizability as maximal pro-$p$ Galois group.

\begin{thm}\label{thm:intro}
Let $G$ be a finitely generated $p$-adic analytic pro-$p$ group.
The following are equivalent:
\begin{itemize}
 \item[(i)] $G$ may be completed into a 1-smooth oriented pro-$p$ pair $\calG=(G,\theta)$ (with $\Img(\theta)\subseteq 1+4\Z_2$, if $p=2$);
 \item[(ii)] $G$ is Bloch-Kato (and moreover $\alpha^2=0$ for every $\alpha\in H^1(G,\Z/2)$, if $p=2$).
 \item[(iii)] $G$ occurs as the maximal pro-$p$ Galois group of a field $\K$ containing a primitive $p$-th root of 1 (and also $\sqrt{-1}$, if $p=2$).
\end{itemize}
\end{thm}

\noindent
(Observe that if $\K$ is a field containing $\sqrt{-1}$, then it is well-known that $\Img(\theta_{\K,2})\subseteq1+4\Z_2$ and $\alpha^2=0$ for every $\alpha\in H^1(G_{\K}(2),\Z/2)$.)

Implication (i)$\Rightarrow$(ii) of Theorem~\ref{thm:intro} gives a positive answer to the Smoothness Conjecture for the class of finitely generated $p$-adic analytic pro-$p$ groups, as a Bloch-Kato pro-$p$ group is --- quite obviously --- also weakly Bloch-Kato.
Thus, Theorem~\ref{thm:intro} provides a concrete example of a class of pro-$p$ groups for which the (weak) Bloch-Kato property follows from 1-smoothness --- other examples are free pro-$p$ groups and Demushkin group.
(After the publication of this result, the Smoothness Conjecture has been proved for the class of {\sl right-angled Artin pro-$p$ groups} by I.~Snopce and P.~Zalesski\u{\i}, see \cite{sz:raags}.)

In fact, analytic pro-$p$ groups represent the ``upper bound'' of the class of Bloch-Kato pro-$p$ groups (i.e., Bloch-Kato pro-$p$ groups for which $H^2(G,\Z/p)$ is as large as possible), while the ``lower bound'' (i.e., $H^2(G,\Z/p)$ is as small as possible) consists of free pro-$p$ groups and Demushkin groups.
Thus, by Theorem~\ref{thm:intro}, for the two opposite ``pillars'' of the class of Bloch-Kato pro-$p$ groups, the Bloch-Kato property follows from 1-smoothness.

Moreover, the structure of torsion-free $p$-adic analytic Bloch-Kato pro-$p$ groups is extremely rigid, and all such pro-$p$ groups occur as maximal pro-$p$ Galois groups of fields (see, e.g., \cite[\S~3.1--3.2]{cmq:fast}).
By Theorem~\ref{thm:intro}, this rigidity in terms of structure follows also from 1-smoothness: this suggests that 1-smoothness is a very strong and restrictive condition.
We believe that a further investigation of 1-smoothness for pro-$p$ groups may lead to the discovery of new obstructions for the structure of maximal pro-$p$ Galois groups --- and absolute Galois pro-$p$ groups --- of fields (see, e.g., \cite{cq:nogal}).

Last, but not least, it is worth mentioning that the class of $p$-adic analytic pro-$p$ groups is an important class of groups to consider --- besides the Bloch-Kato property ---, for the role such groups play in the {\sl $p$-adic Langlands program} (see, e.g., \cite{langlands}).

{\small \begin{rem}\label{rem:intro}\rm
The research carried out in this manuscript was originally made public in the preprint \cite{cq:1smooth}, published on arXiv in April 2019 (in particular, Theorem~\ref{thm:intro} was \cite[Thm.~1.3]{cq:1smooth}), and submitted to a refereed journal. Subsequently, we decided to change strategy, and to split the original paper: this manuscript is one of the two resulting pieces.
In the meanwhile, the research on 1-smooth oriented pro-$p$ groups went on and it lead to other results, such as the aforementioned work by Snopce and Zalesski\u{i} \cite{sz:raags}, and \cite{cq:nogal,BQW:raags}.
In particular, the results contained in \cite{cq:1smooth} have been quoted in the subsequent works \cite{cq:nogal,cq:galfeat,st:fratini}.
        \end{rem}}


\section{Oriented pro-$p$ groups and Kummerianity}
\label{sec:kummerian}

We work in the category of pro-$p$ groups; by an abuse of notation, ``subgroup'' will always mean ``closed subgroup'', and sets of generators of pro-$p$ groups, and presentations, are to be intended in the topological sense.
Therefore, sets of generators of pro-$p$ groups, and presentations, are to be intended in the topological sense.
Given a pro-$p$ group $G$, we denote the closed commutator subgroup of $G$ (i.e., the closed normal subgroup generated by commutators $[g,h]=g^{-1}h^{-1}gh$, $g,h\in G$) by $G'$; the {\sl Frattini subgroup} of $G$ is denoted by $\Phi(G)$ (cf. \cite[Prop.~1.13]{ddsms}).

Recall that $1+p\Z_p=\{1+p\lambda\:\mid\:\lambda\in\Z_p\}$ is a multiplicative abelian pro-$p$ group.
In particular, if $p$ is odd then $1+p\Z_p\simeq \Z_p$ (the latter being considered as an additive pro-$p$ group), and $1+p\Z_p$ is torsion-free; while if $p=2$ then 
\begin{equation}\label{eq:12Z2}
 1+2\Z_p=\{\pm1\}\times(1+4\Z_2)=(\Z/2)\oplus\Z_2
\end{equation}
(the latter being considered as an additive pro-$2$ group).

Following \cite{qw:cyc}, we call a pair $\calG=(G,\theta)$, consisting of a pro-$p$ group $G$ together with a morphism of pro-$p$ groups $\theta\colon G\to1+p\Z_p$, an {\sl oriented pro-$p$ group}, and the morphism $\theta$ is called an {\sl orientation} of $G$. 
(In \cite{efrat:small,eq:kummer}, an oriented pro-$p$ group is called a ``cyclotomic pro-$p$ pair'' --- for the motivation of the name ``orientation'', see the footnote at the end of p.~1885 in \cite{qw:cyc}.)
An orientation $\theta\colon G\to1+p\Z_p$ is said to be {\sl torsion-free} if the group $\Img(\theta)$ is torsion-free  (cf. \cite[\S~2]{eq:kummer}) --- namely, if $p=2$ then by \eqref{eq:12Z2} we require that $\Img(\theta)\subseteq1+4\Z_2$.

An oriented pro-$p$ group $\calG=(G,\theta)$ has a distinguished continuous pro-$p$ (left) $G$-module $\Z_p(\theta)$,
which is equal to the additive group $\Z_p$, and it is endowed with left $G$-action given by
\[
 g\cdot z=\theta(g)\cdot z, \quad\text{for all }g\in G,z\in\Z_p(\theta).
\]
The $G$-module $\Z_p(\theta)/p$ is a trivial $G$-module isomorphic to $\Z/p$, as $\theta(g)\equiv 1\bmod p$ for all $g\in G$.
Similarly, if $p=2$ and $\theta$ is a torsion-free orientation, then $\Z_2(\theta)/4$ is a trivial $G$-module isomorphic to $\Z/4$, as $\theta(g)\equiv 1\bmod 4$ for all $g\in G$.

A morphism of oriented pro-$p$ groups $\calG_1\to\calG_2$, with $\calG_i=(G_i,\theta_i)$ for $i=1,2$,
is a homomorphism of pro-$p$ groups $\phi\colon G_1\to G_2$ such that $\theta_1=\theta_2\circ\phi$ (cf. \cite[\S~3, p.~1888]{qw:cyc}).
In the continuation, we will use the following constructions of oriented pro-$p$ groups.
Let $\calG=(G,\theta)$ be an oriented pro-$p$ group.
\begin{itemize}
 \item[(a)] If $N$ is a normal subgroup of $G$ contained in $\Ker(\theta)$, one has the oriented pro-$p$ group
 \begin{equation}\label{eq:or quotient}
  \calG/N=(G/N,\bar\theta),\end{equation}
 where $\bar\theta\colon G/N\to1+p\Z_p$ is the orientation such that $\bar\theta\circ\pi=\theta$, with $\pi\colon G\to G/N$ the canonical projection.
 \item[(b)] If $A$ is an abelian pro-$p$ group (written multiplicatively), one has the oriented pro-$p$ pair
 \begin{equation}\label{eq:or semidirect}
  A\rtimes\calG=(A\rtimes G,\tilde\theta),
 \end{equation}
with action given by $gag\inv=a^{\theta(g)}$ for every $g\in G$, $a\in A$,
 where the orientation $\tilde\theta\colon A\rtimes G\to1+p\Z_p$ is the composition of the canonical projection $A\rtimes G\to G$ with $\theta$ (this construction was introduced by I.~Efrat in \cite[\S~3]{efrat:small}).
\end{itemize}

\begin{defin}\label{defin:thetabelian}\rm
 An oriented pro-$p$ group $\calG=(G,\theta)$ is said to be {\sl $\theta$-abelian} if $\calG\simeq A\rtimes \calG/\Ker(\theta)$ for some free abelian pro-$p$ group $A$.
\end{defin}

An oriented pro-$p$ group $\calG=(G,\theta)$ has a distinguished subgroup: the subgroup
\begin{equation}\label{eq:KG}
 K(\calG)=\left\langle\:ghg^{-1}h^{-\theta(g)}\:\mid\:g\in G,\:h\in\Ker(\theta)\:\right\rangle
\end{equation}
(cf. \cite[\S~3]{eq:kummer}).
The subgroup $K(\calG)$ is normal in $G$, and moreover one has 
\begin{equation}\label{eq:KG prop}
\Phi(G)\supseteq K(\calG)\qquad\text{and}\qquad \Ker(\theta)\supseteq K(\calG)\supseteq\Ker(\theta)',
\end{equation}
so that $\Ker(\theta)/K(\calG)$ is an abelian pro-$p$ group.
Observe that for every $g\in G$ and $h\in\Ker(\theta)$ one has $ghg^{-1}\equiv h^{\theta(g)}$ modulo $K(\calG)$, and hence 
\begin{equation}\label{eq:G mod KG}
\calG/K(\calG)\simeq \Ker(\theta)/K(\calG)\rtimes\calG/\Ker(\theta)
\end{equation}
in the sense of \eqref{eq:or semidirect}.
Moreover, if $\calG=(G,\theta)$ is a $\theta$-abelian oriented pro-$p$ group, then $K(\calG)=\{1\}$.

The following result gives a group-theoretic characterization of finitely generated Kummerian oriented pro-$p$ groups (cf. \cite[Thm.~5.6 and Thm.~7.1]{eq:kummer}).

\begin{thm}\label{thm:kummer}
Let $\calG=(G,\theta)$ be an oriented pro-$p$ group, with $G$ finitely generated and $\theta\colon G\to1+p\Z_p$ a torsion-free orientation.
The following are equivalent.
\begin{itemize}
 \item[(i)] $\calG$ is Kummerian.
 \item[(ii)] $\Ker(\theta)/K(\calG)$ is a free abelian pro-$p$ group.
 \item[(iii)] $\calG/K(\calG)=(G/K(\calG),\bar\theta)$ is a $\bar\theta$-abelian oriented pro-$p$ group.
\end{itemize} 
\end{thm}

In particular, by \eqref{eq:G mod KG} and Theorem~\ref{thm:kummer}, a finitely generated oriented pro-$p$ group $\calG=(G,\theta)$, with $\theta$ a torsion-free orientation and $K(\calG)=\{1\}$, is Kummerian if, and only if, $\calG$ is $\theta$-abelian.

\begin{rem}\label{rem:abs torfree}\rm
 If $\calG=(G,\mathbf{1})$ is an oriented pro-$p$ group with $\mathbf{1}\colon G\to1+p\Z_p$ the orientation constantly equal to 1, then $K(\calG)=G'$.
 By Theorem~\ref{thm:kummer}, the oriented pro-$p$ group $\calG$ is Kummerian if, and only if, the abelianization $G/G'=\Ker(\mathbf{1})/K(\calG)$ of $G$ is a free abelian pro-$p$ group.
 In particular, if $\calG$ is also 1-smooth, then the abelianization of every finitely generated subgroup of $G$ is a free abelian pro-$p$ group.
\end{rem}

\begin{exam}\label{ex:kummer}\rm
\begin{itemize}
 \item[(a)] Let $G$ be a free pro-$p$ group. Then the oriented pro-$p$ group $\calG=(G,\theta)$ is 1-smooth for any orientation $\theta$ (cf. \cite[\S~2.2]{qw:cyc}).
 \item[(b)] Let $G$ be a Demushkin group (cf., e.g., \cite[Def.~3.9.9]{nsw:cohn}). Then there exists one --- and only one --- orientation $\theta\colon G\to1+p\Z_p$ which completes $G$ into a 1-smooth oriented pro-$p$ group $\calG=(G,\theta)$ (cf. \cite[Thm.~4]{labute:demushkin} and \cite[Cor.~5.7]{qw:cyc}).
  \end{itemize}
\end{exam}

From the following example (cf. \cite[Ex.~3.5]{eq:kummer}), one may recover the Artin-Schreier obstruction as a consequence of 1-smoothness.

 \begin{exam}\label{ex:AS}\rm
For $p$ odd, let $G$ be a finite $p$ group, and let $\calG=(G,\theta)$ be an oriented pro-$p$ group.
Then $\theta\equiv\mathbf{1}$, as $1+p\Z_p$ is torsion-free, and thus $\Ker(\theta)=G$ and $K(\calG)=\{1\}$
Hence $\calG$ is not Kummerian.

Similarly, for $p=2$ let $G$ a group of order 4, and let $\calG=(G,\theta)$ be an oriented pro-$2$ group.
By \eqref{eq:12Z2} $\Ker(\theta)\neq\{1\}$, while $K(\calG)=\{1\}$ (cf. \cite[Ex.~3.5--(4)--(5)]{eq:kummer}).
Hence $\calG$ is not Kummerian.
By contrast, the oriented pro-$2$ group $\calG=(G,\theta)$ with $G\simeq\Z/2$ and $\Img(\theta)=\{\pm1\}$ is Kummerian (and thus 1-smooth).
 \end{exam}

\begin{rem}\label{rem:defin kummer}\rm
 In the original definition given in \cite[Def.~3.4]{eq:kummer}, an oriented pro-$p$ group $\calG=(G,\theta)$ is said to be  Kummerian if the quotient $\Ker(\theta)/K(\calG)$ is torsion-free.
 By Theorem~\ref{thm:kummer} this original definition and the ``cohomological'' definition given in the Introduction --- i.e., the morphism \eqref{eq:def kummer} is surjective for every $n\geq1$ ---, which we use throughout the paper, are equivalent if $G$ is finitely generated.
 In \cite[Thm.~1.2]{cq:detect1cyc} it is shown that these two definitions of Kummerianity are equivalent also in the non-finitely generated case.
 
 Finally, note that in \cite{qw:cyc}, the orientation $\theta$ of a 1-smooth oriented pro-$p$ group $\calG=(G,\theta)$ is said to be {\sl 1-cyclotomic}.
\end{rem}


\section{Bloch-Kato pro-$p$ groups and the Smoothness conjecture}
\label{sec:BK}

Here all graded algebras $\bfA_\bullet=\bigoplus_{n\in\Z} A_n$ over a field $\F$ are assumed to be locally finite-dimensional with $A_n=0$ for $n<0$ and $A_0=\F$.
A graded algebra $\bfA_\bullet$ is called a {\sl quadratic algebra} if it is 1-generated --- i.e., every element is a combination of products of elements of degree 1 ---, and its relations are generated by homogeneous relations of degree 2 (cf. \cite[Ch.~1, \S~2]{pp:quad}).
In other words, one has an isomorphism of graded algebras $\bfT_\bullet(A_1)/I\overset{\sim}{\to}\bfA_\bullet$, where $\bfT_\bullet(A_1)=\bigoplus_{n\geq0}A_1^{\otimes n}$ is the tensor $\F$-algebra generated by $A_1$, and $I$ is a two-sided ideal of $\bfT_\bullet(A_1)$ generated as a two-sided ideal by a subset of $A_1^{\otimes}$.

\begin{exam}\rm
 Let $V$ be a finite-dimensional vector space over $\Z/p$.
 \begin{itemize}
  \item[(a)] The tensor $\Z/p$-algebra $\bfT_\bullet(V)$ is quadratic.
  \item[(b)] The exterior algebra $\bfLam_\bullet(V)$ is quadratic, as $\bfT_\bullet(V)/I\simeq\bfLam_\bullet(V)$, with $I$ the two-sided ideal generated by $\{v\otimes v\:\mid\:v\in V\}\subseteq V^{\otimes2}$.
 \end{itemize}
\end{exam}

\begin{rem}\label{rem:wedgecomm}\rm
 If $\bfA_\bullet=\bigoplus_{n\geq0}A_n$ is a quadratic algebra such that $a^2=0$ for every $a\in A_1$, then one has an epimorphism of quadratic algebras $\bfLam_\bullet(A_1)\twoheadrightarrow\bfA_\bullet$.
\end{rem}

\begin{defin}\label{defin:BK}
\rm Let $G$ be a pro-$p$ group, and let $n\geq1$.
Cohomology classes in the image of the natural cup-product
\[
 H^1(G,\Z/p)\times\ldots\times H^1(G,\Z/p)\overset{\cup}{\longrightarrow} H^n(G,\Z/p)
\]
are called {\sl symbols} (relative to $\Z/p$).
\begin{itemize}
 \item[(i)] If for every open subgroup $U\subseteq G$ every element $\alpha\in H^n(U,\Z/p)$, for every $n\geq1$,
can be written as 
\[
 \alpha=\mathrm{cor}_{V_1,U}^n(\alpha_1)+\ldots+\mathrm{cor}^n_{V_r,U}(\alpha_r),\quad r\geq1,
\]
where $\alpha_i\in H^n(V_i,\Z/p)$ is a symbol and $$\mathrm{cor}_{V_i,U}^n\colon H^n(V_i,\Z/p)\longrightarrow H^n(U,\Z/p)$$
is the {\sl corestriction map} (cf. \cite[Ch.~I, \S~5]{nsw:cohn}), for some open subgroups $V_i\subseteq U$,
then $G$ is called a {\sl weakly Bloch-Kato pro-$p$ group} (cf. \cite[Def.~14.23]{dcf:lift}).
 \item[(ii)] If for every subgroup $H\subseteq G$, the $\Z/p$-cohomology algebra 
 \[\bfH^\bullet(H,\Z/p):=\coprod_{n\geq0}H^n(H,\Z/p),\]
endowed with the cup-product, is a quadratic algebra over $\Z/p$,
then $G$ is called a {\sl Bloch-Kato pro-$p$ group} (cf. \cite{cq:bk}).
\end{itemize}
\end{defin}

Clearly, a Bloch-Kato pro-$p$ group is also weakly Bloch-Kato.

 \begin{examples}\label{ex:BK}\rm
 \begin{itemize}
 \item[(a)] A free pro-$p$ group $G$ is Bloch-Kato, as $H^n(G,\Z/p)=0$ for $n\geq2$,
 and also every subgroup $H\subseteq G$ is a free pro-$p$ group (cf. \cite[Ch.~I, \S~4.2, Cor.~2--3]{serre:galc}).
 \item[(b)] A Demushkin group is Bloch-Kato (cf. \cite[Thm.~6.8]{qw:cyc}).
 In particular, every open subgroup of $G$ is again a Demushkin group (cf. \cite[Thm.~3.9.15]{nsw:cohn}), while every closed non-open subgroup of $G$ is a free pro-$p$ group (cf. \cite[Ch.~I, \S~4.5, Ex.~5--(b)]{serre:galc})
  \end{itemize}
 \end{examples}

Let $\K$ be a field containing a primitive $p$-th root of 1.
By the Norm Residue Theorem, the $\Z/p$-cohomology algebra $\bfH^\bullet(G_{\K},\Z/p)$ of the absolute Galois group $G_{\K}$ is quadratic.
By the Hochschild-Serre exact sequence associated to the short exact sequence of profinite groups
 \[
  \xymatrix{ \{1\}\ar[r] & \Gal(\bar \K_s/\K(p))\ar[r] & G_{\K}\ar[r] & G_{\K}(p)\ar[r] & \{1\} },
 \]
one has an isomorphism of graded $\Z/p$ algebras $\bfH^\bullet(G_{\K}(p),\Z/p)\simeq \bfH^\bullet(G_{\K},\Z/p)$ (cf., e.g., \cite[\S~2]{cq:bk}), so that also $\bfH^\bullet(G_{\K}(p),\Z/p)$ is quadratic.
Thus, $G_{\K}(p)$ is a Bloch-Kato pro-$p$ group.

The following is the pro-$p$ version of the Smoothness Conjecture formulated by C.~De~Clerq and M.~Florence
(cf. \cite[Conj.~14.25]{dcf:lift}).

\begin{conj}\label{conj:weak bk}
Let $\calG=(G,\theta)$ be a 1-smooth oriented pro-$p$ group, with $\theta$ a torsion-free orientation.
Then $G$ is a weakly Bloch-Kato pro-$p$ group.
\end{conj}

A positive answer to the Smoothness Conjecture would provide a new proof of the ``1-generation half'' of the Bloch-Kato conjecture (cf. \cite[\S~1.1]{dcf:lift}), alternative to the proof by Rost and Voevodsky.
Indeed, by Milnor $K$-theory one has that the weak Bloch-Kato property of the maximal pro-$p$ group $G_{\K}(p)$ of a field $\K$, containing a primitive $p$-th root of 1, implies that the algebra $\bfH^\bullet(G,\Z/p)$ is 1-generated (cf. \cite[Rem.~14.26]{dcf:lift}).


\section{Locally uniform pro-$p$ groups}
\label{sec:locunif}

We recall the following definition.

\begin{defin}\label{defin:unif}\rm
Let $G$ be a pro-$p$ group.
\begin{itemize}
 \item[(a)] $G$ is {\sl powerful} if $G'$ is contained in the subgroup of $G$ generated by $\{g^{p^\epsilon}\:\mid\:g\in G\}$, where $\epsilon=2$ if $p=2$, $\epsilon=1$ otherwise.
 \item[(b)] If $G$ is finitely generated, then $G$ is {\sl uniformly powerful} (or simply {\sl uniform}) if $G$ is powerful and torsion-free.
 \item[(c)] $G$ is {\sl locally uniform} if every finitely generated subgroup of $G$ is uniform.
\end{itemize} 
\end{defin}

(For a detailed account on powerful and uniform pro-$p$ groups and their properties we refer to \cite[Ch.~3--4]{ddsms}.)

By Lazard's work \cite{lazard:analytic}, if $G$ is a uniform pro-$p$ group one has an isomorphism of quadratic $\Z/p$-algebras
\begin{equation}\label{eq:cohom lazard}
 \bfLam_\bullet\left(H^1(G,\Z/p)\right)\overset{\sim}{\longrightarrow}\bfH^\bullet(G,\Z/p)
\end{equation}
(cf., e.g., \cite[Thm.~5.1.5]{sw:cohomology}).
Therefore, a finitely generated locally uniform pro-$p$ group is Bloch-Kato.
Moreover, for locally uniform pro-$p$ groups one has the following (cf. \cite[Thm.~A]{cq:bk} and \cite[Prop.~3.5]{cmq:fast}).

\begin{prop}\label{prop:locunif thetabelian}
A pro-$p$ group $G$ is locally uniform if, and only if, there exists a torsion-free orientation $\theta\colon G\to1+p\Z_p$ such that the oriented pro-$p$ group $\calG=(G,\theta)$ is $\theta$-abelian. 
\end{prop}

Consequently, a locally uniform pro-$p$ group may complete into a Kummerian oriented pro-$p$ group, as a $\theta$-abelian oriented pro-$p$ group is Kummerian by Theorem~\ref{thm:kummer}.
In fact, locally uniform pro-$p$ groups are the only uniform pro-$p$ groups which can do this. 

\begin{prop}\label{prop:uniform kummerian}
Let $G$ be a uniform pro-$p$ group.
Then $G$ may complete into a Kummerian oriented pro-$p$ group $\calG=(G,\theta)$ if, and only if, $G$ is locally uniform.\end{prop}

\begin{proof}
By Proposition~\ref{prop:locunif thetabelian}, it is enough to prove the following implication:
if $G$ may complete into a Kummerian oriented pro-$p$ group $\calG=(G,\theta)$, then $G$ is locally uniform.

If $\calG$ is Kummerian, then by Theorem~\ref{thm:kummer} the oriented pro-$p$ group $\calG/K(\calG)=(G/K(\calG),\bar\theta)$ is $\bar\theta$-abelian, and thus $G/K(\calG)$ is locally uniform by Proposition~\ref{prop:locunif thetabelian}.
So, both $G$ and $G/K(\calG)$ are uniform, and by \eqref{eq:cohom lazard} one has
\begin{equation}\label{eq:deg2 cohom unif}
 \begin{split}
   H^2(G,\Z/p)& \simeq \Lambda_2\left(H^1(G,\Z/p)\right),
 \\H^2(G/K(\calG),\Z/p)& \simeq \Lambda_2\left(H^1(G/K(\calG),\Z/p)\right).
  \end{split}
\end{equation}
On the other hand, the canonical projection $G\to G/K(\calG)$ induces maps
$$ \Inf_{G,K(\calG)}^n\colon H^n(G/K(\calG),\Z/p)\longrightarrow H^n(G,\Z/p)$$
for every $n\geq1$ such that 
$$\Inf_{G,K(\calG)}^1(\alpha)\cup\Inf_{G,K(\calG)}^1(\alpha')=\Inf_{G,K(\calG)}^2(\alpha\cup\alpha')$$
for every $\alpha,\alpha'\in H^1(G/K(\calG),\Z/p)$ (cf. \cite[Prop.~1.5.3]{nsw:cohn}).
Moreover, $\Inf_{G,K(\calG)}^1$ is an isomorphism, as $K(\calG)\subseteq \Phi(G)$ (cf. \cite[Ch.~I, \S~4.2, Remark]{serre:galc}).
Therefore, also $\Inf_{G,K(\calG)}^2$ is an isomorphism, and by the 5-term exact sequence in cohomology
\[ \begin{tikzpicture}[descr/.style={fill=white,inner sep=2pt}]
        \matrix (m) [
            matrix of math nodes,
            row sep=3.5em,
            column sep=3.6em,
            text height=1.5ex, text depth=0.25ex
        ]
        {  0 & H^1(G/K(\calG),\Z/p) & H^1(G,\Z/p) & H^1(K(\calG),\Z/p)^G \\
            & H^2(G/K(\calG),\Z/p) & H^2(G,\Z/p) & \\
           };

        \path[overlay,->, font=\scriptsize,>=latex]
        (m-1-1) edge  (m-1-2) 
        (m-1-2) edge node[auto] {$\Inf_{G,K(\calG)}^1$} (m-1-3) 
        (m-1-3) edge node[auto] {$\res^1_{G,K(\calG)}$} (m-1-4)
        (m-1-4) edge[out=355,in=175] node[descr,yshift=0.3ex] {$\mathrm{trg}$} (m-2-2)
        (m-2-2) edge node[auto] {$\Inf_{G,K(\calG)}^2$} (m-2-3);
\end{tikzpicture}\]
(cf. \cite[Prop.~1.6.7]{nsw:cohn})
one has $H^1(K(\calG),\Z/p)^G=0$.
Since $G$ is a pro-$p$ group and $H^1(K(\calG),\Z/p)$ is a $p$-elementary abelian group, this implies that $H^1(K(\calG),\Z/p)=0$, i.e., $K(\calG)$ is trivial, and $G\simeq G/K(\calG)$ is locally uniform.
\end{proof}

\begin{rem}\label{rem:locunif Gal}\rm
 It is well-known that a finitely generated locally uniform pro-$p$ group may be realized as the maximal pro-$p$ Galois group of a field (cf. \cite[Rem.~3.4]{efrat:small}).
 For example, let $\ell$ is a prime number, $\ell\neq p$, and for $k\geq1$ set $\F=\F_\ell(\xi)$, with $\xi\in(\bar\F_\ell)_s$ a root of 1 of order $p^k$.
Let $\K=\F_{\ell^n}(\!(X_1,\ldots,X_d)\!)$ be the field of Laurent series in the indeterminates $X_1,\ldots,X_d$, $d\geq1$, and with coefficients in $\F$.
Then $$\calG_{\K}=(G_{\K}(p),\theta_{\K,p})\simeq\Z_p^d\rtimes\calG_{\K}/\Ker(\theta_{\K,p}),$$ and $\Img(\theta_{\K,p})=1+p^k\Z_p$ (cf. \cite[Ex.~4.10]{cq:bk}).
\end{rem}


\section{$p$-adic analytic pro-$p$ groups}
\label{sec:analytic}

For a pro-$p$ group $G$ let $\dd(G)$ denote the minimal number of generators of $G$, i.e., 
$\dd(G)=\dim(G/\Phi(G))$, and let the {\sl rank} of $G$ be the supremum of all $\dd(H)$ with $H$ running through
all closed subgroups of $G$ (cf. \cite[\S~3.2]{ddsms}).
Then every finitely generated powerful pro-$p$ group has finite rank (cf. \cite[Thm.~3.13]{ddsms}).

The following result defines finitely generated {\sl $p$-adic analytic pro-$p$ groups}
(cf. \cite[Thm.~8.32 and Cor.~8.33]{ddsms}).

\begin{thm}\label{thm:analytic}
Let $G$ be a finitely generated pro-$p$ group. 
The following are equivalent:
\begin{itemize}
 \item[(i)] $G$ is a $p$-adic analytic manifold and the map $(x,y)\mapsto x\inv y$ is analytic;
 \item[(ii)] $G$ contains an open subgroup which is uniformly powerful;
 \item[(iii)] $G$ has finite rank.
\end{itemize}
A finitely generated pro-$p$ groups satisfying the above properties is a $p$-adic analytic pro-$p$ group.
\end{thm}

Hence, a subgroup of a finitely generated $p$-adic analytic pro-$p$ group has finite rank, and thus is $p$-adic analytic.
Moreover, if $N$ is a normal subgroup of a $p$-adic analytic pro-$p$ group $G$, then also $G/N$ has finite rank, and thus it is $p$-adic analytic (cf. \cite[Exercise~3.1]{ddsms}).

The {\sl dimension} $\dim(G)$ of a $p$-adic analytic pro-$p$ group $G$ is the minimal number of generators $\dd(U)$
of a uniform subgroup $U$ of $G$ (by \cite[Lemma~4.6]{ddsms} $\dim(G)$ does not depend on the choice of 
the uniform subgroup).
One has the following (cf. \cite[Thm.~4.8]{ddsms}).

\begin{prop}\label{prop:dim analytic quotient}
 Let $G$ be a $p$-adic analytic pro-$p$ group, and let $N\subseteq G$ be a normal subgroup of $G$.
 Then 
 \begin{equation}\label{eq:dim analytic quotient}
  \dim(G)=\dim(N)+\dim(G/N).
 \end{equation}
\end{prop}

\begin{exam}\rm
\begin{itemize}
 \item[(a)] A finitely generated abelian pro-$p$ group $G$ is $p$-adic analytic. In particular, if $G\simeq\Z_p^n\oplus A$, with $A$ a finite abelian $p$-group, then $\dim(G)=n$. 
 \item[(b)] If $G$ is a finitely generated locally powerful pro-$p$ group, then $G$ is $p$-adic analytic by Theorem~\ref{thm:analytic}, and $\dim(G)=\dd(G)$.
\end{itemize}
\end{exam}

\begin{exam}\label{exam:heisenberg}\rm
Let $p$ be a odd prime.
 The {\sl Heisenberg group over $\Z_p$} is the group $G$ of upper uni-triangular matrices over $\Z_p$, and it is a torsion-free $p$-adic analytic pro-$p$ group of dimension 3 (cf. \cite[Thm.~7.4--(2)]{GSK}).
In particular, $G$ has a presentation
\[
 G=\langle \:x,y,z\:\mid\:[x,y]=z,\:[x,z]=[y.z]=1\:\rangle,
\]
and one has $G/G'\simeq \Z_p^2$ and $G'=\langle z\rangle\simeq \Z_p$.
Thus, the oriented pro-$p$ group $\calG_{\mathbf{1}}=(G,\mathbf{1})$ is Kummerian by Remark~\ref{rem:abs torfree}.
Set $t=x^p$, and let $U$ be the subgroup of $G$ generated by $t,y,z$.
Then 
\[
 U=\left\langle\:t,y,z\:\mid\:[t,y]=z^p,\:[t,z]=[y,z]=1\:\right\rangle
\]
(cf. \cite[Ex.~7.2]{GSN}).
Hence, $U$ is uniform, and consequently $\dim(U)=\dd(U)=3$.
Yet, $U$ is not locally uniform, and therefore $U$ cannot complete into a Kummerian oriented pro-$p$ group by Proposition~\ref{prop:uniform kummerian}.
Altogether, $G$ cannot complete into a 1-smooth oriented pro-$p$ group.
\end{exam}

\begin{prop}\label{prop:analytic theta1}
Let $G$ be a finitely generated $p$-adic analytic pro-$p$ group, and suppose that the oriented pro-$p$ group $\calG=(G,\mathbf{1})$, with $\mathbf{1}\colon G\to1+p\Z_p$ the orientation constantly equal to 1, is 1-smooth.
Then $G$ is a free abelian pro-$p$ group.
\end{prop}

\begin{proof}
Since $G$ is $p$-adic analytic, every subgroup of $G$ is finitely generated by Theorem~\ref{thm:analytic}.
Thus, by Remark~\ref{rem:abs torfree} every subgorup of $G$ has torsion-free abelianization, i.e., $G$ is an {\sl absolutely torsion-free} pro-$p$ group (absolutely torsion free pro-$p$ groups were introduced by T.~W\"urfel in \cite{wurfel}).

Let $G^{(n)}$, $n\geq1$, denote the derived series of $G$, i.e., $G^{(1)}=G$ and $G^{(n+1)}=[G^{(n)},G^{(n)}]$.
Since $G$ is a finitely generated $p$-adic analytic pro-$p$ group, also the subgroups $G^{(n)}$ and the quotients $G^{(n)}/(G^{(n)})'=G^{(n)}/G^{(n+1)}$ are finitely generated $p$-adic analytic pro-$p$ groups.
Moreover, since $G$ is absolutely torsion-free, one has
\begin{equation}\label{eq:quot series}
G^{(n)}/G^{(n+1)}=G^{(n)}/(G^{(n)})'\simeq\Z_p^{\dd(G^{(n)})}\qquad \text{for all }n\geq1.
\end{equation}
Consequently, $\dim(G^{(n)}/G^{(n+1)})=\dd(G^{(n)})$.
From Proposition~\ref{prop:dim analytic quotient} and from \eqref{eq:quot series}, one deduces 
\begin{equation}
   \dim(G^{(n+1)}) = \dim(G^{(n)})-\dd(G^{(n)}).
   \end{equation}
Since $\dim(G)$ is finite, one has $\dim(G^{(n)})=0$ for some $n$.
Again by Proposition~\ref{prop:dim analytic quotient}, this implies that $\dim(G^{(n)}/(G^{(n)})')=\dd(G^{(n)})=0$, i.e., $G^{(n)}=\{1\}$.
This proves that $G$ is a solvable pro-$p$ group.
By \cite[Prop.~2]{wurfel}, an absolutely torsion-free solvable pro-$p$ group is a free abelian pro-$p$ group, and this concludes the proof.
\end{proof}

\begin{prop}\label{prop:abel ker}
 Let $\calG=(G,\theta)$ be a 1-smooth oriented pro-$p$ group with $\theta$ a torsion-free orientation.
 If $\Ker(\theta)$ is abelian, then $\calG$ is $\theta$-abelian. 
\end{prop}

\begin{proof}
 If the orientation $\theta$ is constantly equal to 1, then $\Ker(\theta)=G$.
 Thus, by Remark~\ref{rem:abs torfree} $G=G/G'$ is a free abelian pro-$p$ group, so that $\calG$ is $\theta$-abelian.
 
 Suppose now that $\theta\not\equiv\mathbf{1}$.
We assume first that $p\neq2$.
 Pick two arbitrary elements $x,y\in G$ such that $\theta(x)\neq 1$ and $y\in \Ker(\theta)$, and put $z=[x,y]$ and $t=y^p$.
Clearly, $z,t\in\Ker(\theta)$.
Since $z,y\in\Ker(\theta)$, which is abelian by hypothesis, one has $z^y=z$, and hence commutator calculus yields
\begin{equation}\label{eq:1proof}
 [x,t]=[x,y^p]=z\cdot z^y\cdots z^{y^{p-1}}=z^p.
\end{equation}
Let $H$ be the subgroup of $G$ generated by $x,y$, and let $U$ be the subgroup of $H$ generated by $x,t,z$.
Then the oriented pro-$p$ groups $\calG_H=(H,\theta\vert_H)$ and $\calG_U=(U,\theta\vert_U)$ are 1-smooth.

Put $\lambda=1-\theta(x)^{-1}$.
Then $0\neq\lambda\in p\Z_p$, as $1\neq\theta(x)^{-1}\in 1+p\Z_p$.
By definition, $[x,t]\cdot t^{-\lambda}\in K(\calG_U)$.
Since $t$ and $z$ commute, from \eqref{eq:1proof} one deduces 
\begin{equation}\label{eq:2proof}
\left(zt^{-\lambda/p}\right)^p=z^pt^{-\frac{\lambda}{p}p}=z^pt^{-\lambda}= [x,t]t^{-\lambda}
 \in K(\calG_U).
\end{equation}
Moreover, $zt^{-\lambda/p}\in\Ker(\theta\vert_U)$.
Since $\calG_U$ is 1-smooth (and thus Kummerian), by Theorem~\ref{thm:kummer} the quotient $\Ker(\theta\vert_U)/K(\calG_U)$ is a free abelian pro-$p$ group, and therefore \eqref{eq:2proof} implies that also $zt^{-\lambda/p}$ is an element of $K(\calG_U)$.

Since $K(\calG_U)\subseteq \Phi(U)$, one has $z\equiv t^{\lambda/p}\bmod\Phi(U)$.
Then by \cite[Prop.~1.9]{ddsms}, $U$ is generated by $x$ and $t$.
Since $[x,t]\in U^p$ by \eqref{eq:1proof}, the pro-$p$ group $U$ is powerful --- and hence uniformly powerful, as it is torsion-free (cf. Example~\ref{ex:AS}).
Therefore, $\calG_U$ is $\theta_U$-abelian by Proposition~\ref{prop:uniform kummerian}.
In particular, $K(\calG_U)=\{1\}$ by \eqref{eq:G mod KG} and Theorem~\ref{thm:kummer}, and thus
\begin{equation}\label{eq:boh}
 [x,y]=z=t^{\lambda/p}=y^{1-\theta(x)^{-1}}.
\end{equation}
Since $\Ker(\theta)$ is abelian by hypothesis, and since $x\in G\smallsetminus\Ker(\theta)$ and $y\in\Ker(\theta)$ were arbitrarily chosen, \eqref{eq:boh} implies that
$\calG\simeq\Ker(\theta)\rtimes \calG/\Ker(\theta)$ in the sense of \eqref{eq:or semidirect}.
Since $\Ker(\theta)$ is torsion-free (cf. Example~\ref{ex:AS}), $\calG$ is $\theta$-abelian.

Finally, assume that $\theta\not\equiv\mathbf{1}$ and $p=2$. 
Since $\calG$ is torsion-free, $\Img(\theta)\subseteq1+4\Z_2$, and the above argument works verbatim if one replaces $p$ with 4: indeed, one has $0\neq\lambda\in 4\Z_2$, as $1\neq\theta(x)^{-1}\in1+4\Z_2$, and $[x,t]\in U^4$, so that the pro-$2$ group $U$ is powerful also in this case.
Hence, $\calG$ is a $\theta$-abelian oriented pro-2 group.
\end{proof}

\begin{thm}\label{thm:analytic 1smooth}
 Let $\calG=(G,\theta)$ be an oriented pro-$p$ group with $G$ a finitely generated $p$-adic analytic pro-$p$ group and $\theta$ a torsion-free orientation.
If $\calG$ is 1-smooth, then it is $\theta$-abelian. 
\end{thm}

\begin{proof}
Since $G$ is $p$-adic analytic, also $\Ker(\theta)$ is $p$-adic analytic.
Since the oriented pro-$p$ group $\calG_{\Ker(\theta)}=(\Ker(\theta),\mathbf{1})$ is 1-smooth, Proposition~\ref{prop:analytic theta1} implies that $\Ker(\theta)$ is a free abelian pro-$p$ group.
Thus, Proposition~\ref{prop:abel ker} implies the claim.
\end{proof}

Let $p=2$, and let $G$ be a pro-2 group.
Also, let $\Z/4$ be a trivial $G$-module.
The short exact sequence of trivial $G$-modules
\[
 \xymatrix{ 0\ar[r]  & \Z/2\ar[r]^-{2\cdot} & \Z/4\ar[r]^-{\pi} &\Z/2\ar[r] &  0 }
\]
induces an exact sequence in cohomology
\begin{equation}\label{eq:les Bockstein}
 \begin{tikzpicture}[descr/.style={fill=white,inner sep=2pt}]
        \matrix (m) [
            matrix of math nodes,
            row sep=3em,
            column sep=3em,
            text height=1.5ex, text depth=0.25ex
        ]
        {   & H^1(G,\Z/2) & H^1(G,\Z/4) & H^1(G,\Z/2) \\
            & H^2(G,\Z/2) & H^2(G,\Z/4) &\cdots  \\
           };

        \path[overlay,->, font=\scriptsize,>=latex]
        (m-1-2) edge node[auto] {} (m-1-3) 
        (m-1-3) edge node[auto] {$\pi^\ast$} (m-1-4)
        (m-1-4) edge[out=355,in=175] node[descr,yshift=0.3ex] {$\mathfrak{b}$} (m-2-2)
        (m-2-2) edge node[auto] {} (m-2-3)
        (m-2-3) edge node[auto] {} (m-2-4);
\end{tikzpicture}
\end{equation}
and the connecting homomorphism $\mathfrak{b}$ is called the {\sl Bockstein morphism}.
Clearly, the map $\mathfrak{b}$ is trivial if, and only if, the map $\pi^\ast\colon H^1(G,\Z/4)\to H^1(G,\Z/2)$ is surjective.
Moreover, the map $\mathfrak{b}$ is trivial if, and only if $\alpha^2=0$ for every $\alpha\in H^1(G,\Z/2)$ (cf. \cite[Lemma~2.4]{em}).

\begin{rem}\label{rem:torsion Bockstein}\rm
 Set $p=2$.
\begin{itemize}
 \item[(i)] Let $\K$ be a field containing $\sqrt{-1}$. 
Then $\Img(\theta_{\K,2})\subseteq1+4\Z_2$, and $\Z_2(\theta_{\K,2})/4$ is isomorphic to $\Z/4$ as a (trivial) $G_{\K}(2)$-module.
Since the oriented pro-$2$ group $\calG_{\K}=(G_{\K}(2),\theta_{\K,2})$ is Kummerian, the map 
$$\pi^\ast\colon H^1(G_{\K}(2),\Z/4)\longrightarrow H^1(G_{\K}(2),\Z/2)$$ is surjective, and thus $\mathfrak{b}$ is trivial.
 \item[(ii)] Let $G$ be a pro-2 group.
If $\bfH^\bullet(G,\Z/2)$ is a quadratic $\Z/2$-algebra and the Bockstein morphism $\mathfrak{b}$ is trivial, then by Remark~\ref{rem:wedgecomm} one has an epimorphism of quadratic $\Z/2$-algebras 
$$\xymatrix{\bfLam_\bullet\left(H^1(G,\Z/2)\right)\ar@{->>}[r] & \bfH^\bullet(G,\Z/2)}.$$
Hence, $\mathrm{cd}(G)\leq \dim(H^1(G,\Z/2))$ (here $\mathrm{cd}(G)$ denotes the cohomological dimension, cf. \cite[Def.~3.3.1]{nsw:cohn} ).
Consequently, $G$ is torsion-free, as a pro-$p$ group with non-trivial torsion has infinite cohomological dimension.
\end{itemize} 
\end{rem}

\begin{cor}\label{cor:analytic}
Let $G$ be a finitely generated $p$-adic analytic pro-$p$ group.
The following are equivalent.
\begin{itemize}
 \item[(i)] $G$ may be completed into a 1-smooth oriented pro-$p$ group $\calG=(G,\theta)$ with $\theta$ a torsion-free orientation.
 \item[(ii)] $G$ is a Bloch-Kato pro-$p$ group, and the Bockstein morphism $\mathfrak{b}$ is trivial if $p=2$.
 \item[(iii)] $G$ occurs as the maximal pro-$p$ Galois group of a field $\K$ containing a primitive $p$-th root of 1 (and also $\sqrt{-1}$ if $p=2$). 
\end{itemize}
\end{cor}

\begin{proof}
Let $G$ be a finitely generated $p$-adic analytic pro-$p$ group.
First, we show that each of the three conditions implies that $G$ may be completed into a $\theta$-abelian oriented pro-$p$ group $\calG=(G,\theta)$ with $\theta$ a torsion-free orientation.
Then, we show that if $\calG=(G,\theta)$ is a $\theta$-abelian oriented pro-$p$ group with $\theta$ torsion-free, then all three conditions (i), (ii), (iii) hold.

If $G$ may be completed into a 1-smooth oriented pro-$p$ group $\calG=(G,\theta)$ with $\theta$ a torsion-free orientation, then $\calG$ is $\theta$-abelian by Theorem~\ref{thm:analytic 1smooth}.
On the other hand, if $G$ is a Bloch-Kato pro-$p$ group (satisfying the further condition $\mathfrak{b}\equiv\mathbf{0}$ if $p=2$), then $G$ may be completed into a $\theta$-abelian oriented pro-$p$ group $\calG=(G,\theta)$ by \cite[Thm.~4.6]{cq:bk} if $p\neq2$ and by \cite[Thm.~4.11]{cq:bk} if $p=2$ (note that in this case $G$ is torsion-free by Remark~\ref{rem:torsion Bockstein}--(ii)).
Moreover, if $G\simeq G_{\K}(p)$ for some field containing a primitive $p$-th root of 1 (and $\sqrt{-1}$ if $p=2$) then $G$ is a Bloch-Kato pro-$p$ group by the Norm Residue Theorem (and by Remark~\ref{rem:torsion Bockstein}--(i) $\mathfrak{b}\equiv\mathbf{0}$ if $p=2$), so that (iii) implies (ii).

Conversely, if $\calG=(G,\theta)$ is a $\theta$-abelian oriented pro-$p$ group with $\theta$ torsion-free, then $G$ is a finitely generated locally uniform pro-$p$ group by Proposition~\ref{prop:locunif thetabelian}.
Therefore: (i) for every subgroup $H$ of $G$, the oriented pro-$p$ group $\calG_H=(H,\theta\vert_H)$ is Kummerian by Theorem~\ref{thm:kummer}, and thus $\calG$ is 1-smooth; 
(ii) $G$ is a Bloch-Kato pro-$p$ group by \eqref{eq:cohom lazard} --- and moreover $\mathfrak{b}\equiv\mathbf{0}$ as $H^2(G,\Z/2)\simeq\Lambda_2(H^1(G,\Z/2))$, if $p=2$; (iii) $G$ occurs as the maximal pro-$p$ Galois group of a field containing a primitive $p$-th root of 1 by Remark~\ref{rem:locunif Gal}.
\end{proof}

Corollary~\ref{cor:analytic} implies Theorem~\ref{thm:intro}.
As mentioned in the Introduction, this result is particularly relevant because
$p$-adic analytic Bloch-Kato pro-$p$ groups are the ``upper bound'' of the class of Bloch-Kato pro-$p$ groups, in the following sense:
if a finitely generated (non-trivial) pro-$p$ group $G$ is Bloch-Kato,
then by \cite[Prop.~4.1]{cq:bk} for the cohomological dimension $\mathrm{cd}(G)$
and the number of defining relations $\mathrm{r}(G)$ --- the latter being equal to $\dim H^2(G,\Z/p)$
(cf. \cite[Ch.~I, \S~4.3]{serre:galc}) --- one has bounds
\[   1\leq \mathrm{cd}(G)\leq \dd(G)\qquad\text{and}\qquad 0\leq\mathrm{r}(G)\leq\binom{\dd(G)}{2}.\]
The lower bounds occur if $G$ is a free pro-$p$ group (and thus $G$ is 1-smooth, cf. Example~\ref{ex:kummer}--(a)).
The upper bounds occur when $G$ is $p$-adic analytic.
In particular, if $G$ is a finitely generated Bloch-Kato pro-$p$ group
(satisfying $\mathfrak{b}\equiv\mathbf{0}$, if $p=2$), the following three conditions are equivalent: (i) $\mathrm{cd}(G)=\dd(G)$; (ii) $\mathrm{r}(G)=\binom{\dd(G)}{2}$; (iii) $G$ is $p$-adic analytic (cf. \cite[Cor.~4.8]{cq:bk}).

We conclude with the following remark, which states two open questions on $1$-smooth oriented pro-$p$ groups.

\begin{rem}\rm
\begin{itemize}
 \item[(i)] Bloch-Kato pro-$p$ groups satisfy the following {\sl Tits' alternative}: if a Bloch-Kato pro-$p$ group $G$ is not locally uniform, then it contains a non-abelian free subgroup (cf. \cite[Thm.~B]{cq:bk}).
 In \cite{cq:galfeat}, we conjecture that 1-smooth oriented pro-$p$ groups satisfy the same alternative: namely, if a 1-smooth oriented pro-$p$ group $\calG=(G,\theta)$ is not $\theta$-abelian, then $G$ contains a non-abelian free subgroup.
 \item[(ii)] Torsion-free $p$-adic analytic pro-$p$ groups $G$ are {\sl Poincar\'e duality pro-$p$ groups} of cohomological dimension
$\mathrm{cd}(G)=\dim(G)$  (cf. \cite[\S~5]{sw:cohomology}).
On the opposite side there are Poincar\'e duality pro-$p$ groups of cohomological dimension $\mathrm{cd}(G)=2$,
namely, infinite Demushkin groups, which are both 1-smooth and Bloch-Kato by Examples~\ref{ex:kummer}--(b) and \ref{ex:BK}--(b).
This raises the following sub-question of Conjecture~\ref{conj:weak bk}:
are 1-smooth Poincar\'e duality pro-$p$ groups (weakly) Bloch-Kato?
\end{itemize}

\end{rem}

\medskip 
{\small
\begin{acknowledgment} \rm
The author is deeply indebted with: N.~D.~T\^an, who pointed out to the author the possible importance
of \cite[Prop.~6]{labute:demushkin}, some years ago; P.~Guillot, for the inspiring discussions on the paper \cite{dcf:lift}; 
I.~Efrat and Th.~Weigel, for working with the author on the papers \cite{eq:kummer} and \cite{qw:cyc}
respectively; G.~Chinello, for his helpful comments; and the anonymous referee and the managing editor, for their careful work with this paper.
Moreover, the author wishes to thank also the two anonymous referees who dealt with the original version of the manuscript \cite{cq:1smooth} submitted to another journal (see Remark~\ref{rem:intro}), as their (sometimes diverging) comments contributed to the improvement of this paper.

This paper was inslpired also by the discussions during the workshop ``Nilpotent Fundamental Groups'' which took place at the
Banff International Research Station (Canada) in June 2017,
(see \cite{birs}*{\S~3.1.6, 3.2.6}), so the author is gratefully indebted with the organizers and the participants of the workshop.

\end{acknowledgment}}


\end{document}